\documentclass[11pt]{article}
\usepackage{geometry}                
\usepackage{amsfonts}
\usepackage{amsmath}
\usepackage{amsthm}
\usepackage{amssymb}
\usepackage{tikz}

\tikzstyle{vertex}=[circle, draw, inner sep=0pt, minimum size=6pt]
\newcommand{\vertex}{\node[vertex][fill]}

\usepackage{enumerate}
\theoremstyle{definition}
\theoremstyle{definition}
\theoremstyle{theorem}\newtheorem{prop}{Proposition}
\theoremstyle{definition}
\theoremstyle{definition}
\theoremstyle{definition}
\theoremstyle{definition}\newtheorem*{problem}{Problem}
\theoremstyle{definition}
\theoremstyle{theorem}\newtheorem{thm}{Theorem}
\theoremstyle{theorem}\newtheorem{lemma}{Lemma}
\theoremstyle{definition}
\theoremstyle{definition}
\theoremstyle{definition}
\theoremstyle{theorem}

\newcommand{\R}{\mathbb{R}}

\newcommand{\Z}{\mathbb{Z}}

\newcommand{\D}{\Delta}

\newcommand{\p}{^\prime}

\author{Sachi Hashimoto\\ 
\small University of Chicago \\
\small sachi@uchicago.edu}

\title{Sharper Lower Bounds in the Maximum Degree and Diameter Bounded Subgraph Problem in the Mesh}

\begin{document}
\maketitle

\begin{abstract}

The Maximum Degree and Diameter Bounded Subgraph Problem (MaxDDBS) asks: given a host graph $G$, a bound on maximum degree $\D$, and a diameter $D$, what is the largest subgraph of the host graph with degree bounded by $\D$ and diameter bounded by $D$?  In this paper, we investigate this problem when the host graph is the $k$-dimensional mesh.  We provide lower bounds for the size of the largest subgraph of the mesh satisfying MaxDDBS for all $k$ and $\D \geq 4$ that agree with the known upper bounds up to the first two terms, and show that for $\D = 3$, the lower bounds are at least the same order of growth as the upper bounds.

\end{abstract}

\section{Introduction}

Let $G$ be a connected simple undirected graph (called the \textit{host} graph), and $\D$ and $D$ be positive integers.  Then we can ask the following question.

\begin{problem}[MaxDDBS] What is the largest subgraph of $G$ with maximum degree at most $\D$ and diameter bounded by $D$?

\end{problem}

The size of a subgraph of $G$ is measured by the number of vertices in the subgraph.  

MaxDDBS has applications to parallel computing.  Given some host network, we may want to find the largest subnetwork subject to certain constraints.  In a physical system, there is a limit to how many connections we can attach to a single node, so we consider only networks that have bounded degree to be feasible.  Additionally, bounding the diameter bounds the distance between nodes in the subnetwork, which translates to bounding the distance communications may have to travel within the network.  While there exist a range of problems that consider other constraints, we will focus on just bounding the degree and diameter in this paper, as in \cite{Dekker} and \cite{Miller}.  Reference \cite{Dekker} mentions other problems relating to finding subnetworks with certain interesting constraints, and also discusses the potential applications of MaxDDBS in more depth.

MaxDDBS is a natural generalization of the Degree Diameter Problem, DDP, which asks for the largest graph with given degree and diameter.  In particular, when $G = K_n$, MaxDDBS is DDP.  DDP dates back to the 1964 in a paper by B. Elspas \cite{Elspas}.  Despite this, MaxDDBS is a relatively new problem: it was introduced in a 2011 paper by A. Dekker, H. Perez-Roses, G. Pineda-Villavicencio, and P. Watters \cite{Dekker}.  In that paper, they analyze MaxDDBS for the hypercube and the mesh.  In the case of the hypercube, they make use of the Hamming distance in the cube to provide a lower bound of $\sum_{i=0}^D \binom{\D}{i}$ for the largest subgraph of the $k$-dimensional hypercube with maximum degree $\D$ and diameter $D$.

For the mesh, they use the $\ell^1$ metric on $G= \Z^k$ and split the problem into two cases, letting $N^e_k(\D, p)$ be the number of vertices in the largest subgraph of $k$-dimensional mesh with diameter $2p$ and maximum degree $\D$, and $N^o_k(\D,p)$ the number of vertices in the largest subgraph of $k$-dimensional mesh with diameter $2p+1$ and maximum degree $\D$.  Then, since $G$ is regular with degree $2k$, we have a lower bound on $N^e_k(\D,p)$ which is the $\ell^1$ ball of diameter $2p$ in dimension $\lfloor \D/2 \rfloor$, denoted $B^e_{\lfloor \D/2\rfloor}(p)$.  In the even diameter case, this ball contains the most lattice points when centered at a lattice point.  Similarly, $N^o(\D, p)$ has a lower bound of the $\ell^1$ ball of diameter $2p+1$, $B^o_{\lfloor \D/2\rfloor}(p)$, which contains the most lattice points when centered halfway between two adjacent lattice points.  They also assert without proof that $B^e_k(p)$ and $B^o_k(p)$ are upper bounds on $N^e_k(\D, p)$ and $N^o_k(\D, p)$, respectively, giving the following inequalities.

\begin{prop} [A. Dekker, H. Perez-Roses, G. Pineda-Villavicencio, and P. Watters]\footnotemark \footnotetext{This proposition appears in both \cite{Miller} and \cite{Dekker} but so far we know of no proof for the upper bounds.  Note that it is possible to have a subgraph of $k$-dimensional mesh where each vertex is at most distance $2r$ from every other vertex, but the subgraph is not contained in any $\ell^1$ ball of radius $r$.  For example the cube in 3-dimensions of side length 2 with one vertex in each pair of opposite corners removed is one such graph.}

\[ |B^e_{\lfloor \D/2 \rfloor }(p)| \leq N^e_{k } (\D, p) \leq | B^e_{k} (p)|\] 
\[ |B^o_{\lfloor \D/2 \rfloor}(p)| \leq N^o_{k } (\D, p) \leq | B^o_{k} (p)|\] 

where $\D \leq 2k$ and the absolute value denotes number of lattice points.

\end{prop}

A 2012 paper by M. Miller, H. Perez-Roses, and J. Ryan \cite{Miller} focuses on the case of the $k$-dimensional mesh in more detail, improving the proposed bounds in \cite{Dekker}:  when $\D = 4$, they construct a subgraph of the $3$-dimensional mesh with diameter $D$ which agrees with the asserted upper bounds in Proposition 1 in the first two terms.  They pose the problem of generalizing this construction in $k$ dimensions for $\D = 2k-2$.  We use a pared down version of the constructions in \cite{Miller} to obtain lower bounds for all dimensions and for a constant $\D = 4$ that agree with the upper bounds in Proposition 1 in the first two terms.  In \cite{Miller} they also construct subgraphs of the $2$-dimensional mesh with $\D = 3$ that provide lower bounds which agree with the first term of the asserted upper bounds.  We construct subgraphs of $k$-dimensional mesh with $\D = 3$ that are the same order as the asserted upper bounds.

In light of Proposition 1, it makes sense to consider the values of $|B^e_k(p)|$ and $|B^o_k(p)|$.

\begin{prop}[M. Miller, H. Perez-Roses, and J. Ryan, \cite{Miller}]

\[ |B^e_k(p)| = \sum_{i=0}^k 2^i \binom{k}{i} \binom{p}{i} = \frac{2^k p^k}{k!} +  \frac{2^{k-1} p^{k-1}}{(k-1)!} + O(p^{k-2}) .\] 

\[ |B^o_k(p)| = \sum_{i = 0}^k 2^i \left[ \binom{k}{i} + \binom {k-1}{i} \right] \binom{p}{i}= \frac{2^k p^k}{k!} + \frac{2^k p^{k-1}}{(k-1)!}  + O(p^{k-2}).\]

\end{prop}

In Proposition 1, the lower bounds for $N_{k}^e(\D, p)$ and $N_{k}^o(\D, p)$ are  $\Theta(p^{\lfloor \D/2 \rfloor})$ and the upper bound is $\Theta(p^{k})$.  In general, these are not very close, as $\lfloor\D/2 \rfloor \leq k$.  In fact, we will show that the actual values of $N^e_k (\D, p)$ and $N^o_k (\D, p)$ are much closer to the upper bounds for $\D \geq 3$.  In particular, we prove the following theorems.

\begin{thm}

We have the following bounds on $N^e_k (\D, p)$:

\begin{enumerate}

\item $N^e_k (\D, p) = 2$ when $\D =1$;

\item $N^e_k(\D, p) = 4p$ when $\D = 2$;

\item $N^e_k(\D, p) = \Theta(p^k)$ when $\D = 3$;

\item $N^e_k(\D, p) = \frac{2^k p^k}{k!} +  \frac{2^{k-1} p^{k-1}}{(k-1)!} + O(p^{k-2})$ when $\D \geq 4$.

\end{enumerate}

\end{thm}

\begin{thm}

We have the following bounds on $N^o_k (\D, p)$:

\begin{enumerate}

\item $N^o_k (\D, p) = 2$ when $\D =1$;

\item $N^o_k(\D, p) = 4p+2$ when $\D = 2$;

\item $N^o_k(\D, p) = \Theta(p^k)$ when $\D = 3$;

\item $N^o_k(\D, p) =\frac{2^k p^k}{k!} + \frac{2^k p^{k-1}}{(k-1)!}  + O(p^{k-2})$ when $\D \geq 4$.
\end{enumerate}

\end{thm}

Our proof improves the lower bound by constructing an example of a $k$-dimensional subgraph satisfying the constraints on degree and diameter.

In Section 2, we provide constructions of $k$-dimensional graphs with maximum degree $4$ and diameter $2p$ which provide the lower bounds for $N^e_k(\D, p)$ and $N^o_k(\D, p)$ in the $\D = 4$ case of Theorems 1 and 2.  In Section 3, we start by providing proofs of the $\D = 1$ and $\D =2$ cases, and construct a $k$-dimensional subgraph with maximum degree 3 to prove the $\D = 3$ case of Theorems 1 and 2.

\section{Subgraphs of $k$-dimensional Mesh for $\D \geq 4$}

In this section we look at subgraphs of the $k$-dimensional mesh with maximum degree $\D = 4$.  In order to simplify the coordinates in our proofs, we will embed $G$ isometrically in $\R^k$ two ways: in the even diameter case, we define $G^e$ to have the vertices in the lattice $\Z^k$, and in the odd diameter case, we define $G^o$ to have the vertices in the lattice $(\Z+\frac{1}{2})\times \Z^{k-1} $.  Two points in $G^e$ or $G^o$ share an edge if and only if they are distance one from each other under the $\ell^1$ norm.  We let $(x_1, \dots, x_k)$ denote coordinates in $\Z^k$ or $(\Z+\frac{1}{2})\times \Z^{k-1} $, and in the latter note that $x_1$ will always be the noninteger dimension.

Our goal in these constructions is to make subgraphs of the $k$-dimensional mesh that contain almost all of the vertices in the $\ell^1$ ball of radius $p$.  We can think of the $\ell^1$ ball in $k$ dimensions as being built up inductively from $(k-1)$-dimensional cross sections that are $\ell^1$ balls in $k-1$ dimensions of varying diameter.  Note that if we have a construction that contains all but $O(p^{\alpha-1})$ vertices of the $(k-1)$-dimensional $\ell^1$ ball $B^e_{k-1}(p)$, we can stack copies of this construction of diameter $2(p-i)$ located at $x_k= \pm i$ and get a construction with all but $O(p^{\alpha})$ vertices of the $B_k^e(p)$.  As our aim in this section is to prove part 4 of Theorems 1 and 2, we restate these below.

\[ N^e_k(\D,k) =\frac{ 2^kp^k}{k!} + \frac{2^{k-1}p^{k-1}}{(k-1)!} + O (p^{k-2}), \text{ for } \D \geq 4. \tag{Theorem 1, Part 4}\]

 \[ N^o_k(\D,k) = \frac{2^k p^k}{k!} +  \frac{2^k p^{k-1}}{(k-1)!} + O (p^{k-2}), \text{ for } \D \geq 4. \tag{Theorem 2, Part 4} \]

\subsection{The Even Diameter Case}

In our construction that proves the lower bounds in the even diameter case, $N^e_k(\D, p)$, we will build two graphs, $\mathcal{E}_k(p)$ and $\mathcal{E}\p_k(p)$.  The lower bound will come from $\mathcal{E}\p_k(p)$, but both $\mathcal{E}_{k-1}(p)$ and $\mathcal{E}\p_{k-1}(p)$ will be used in building $\mathcal{E}\p_k (p)$.

\begin{prop} There exists a graph $\mathcal{E}_{k}(p)$ centered at the origin satisfying the following conditions:

\begin{enumerate}

\item The degree of any vertex $v= (v_1, \dots, v_{k})$ which is not the origin is 4 if $v_i = 0$ for some $i$ and $2$ or 1 otherwise;

\item The origin, $(0,0, \dots, 0) \in \Z^{k}$ has degree $2$;

\item Any vertex in $\mathcal{E}_{k}(p)$ is distance at most $p$ from the vertex $(0,0, \dots, 0) \in \Z^{k}$;

\item $|\mathcal{E}_{k}(p)| = \frac{2^{k}p^{k}}{k!} + O (p^{k-1}).$ 

\end{enumerate}

\label{Ekp}
\end{prop}

\begin{figure}
\begin{centering}

\begin{tikzpicture}[scale=0.5]

\vertex[shape=rectangle] (0,0) at (0,0) {};
\vertex (1,0) at (1,0) {};
\vertex (-1,0) at (-1,0) {};

\path
	(0,0) edge (1,0)
	(0,0) edge (-1,0)
	;

\end{tikzpicture} \,\,\,
\begin{tikzpicture}[scale=0.5]

\vertex[shape=rectangle] (0,0) at (0,0) {};
\vertex (1,0) at (1,0) {};
\vertex (-1,0) at (-1,0) {};
\vertex (2,0) at (2,0) {};
\vertex (-2,0) at (-2,0) {};

\path
	(0,0) edge (1,0)
	(0,0) edge (-1,0)
	(1,0) edge (2,0)
	(-1, 0) edge (-2, 0)
	;

\end{tikzpicture} \,\,\,
\begin{tikzpicture}[scale=0.5]

\vertex[shape=rectangle] (0,0) at (0,0) {};
\vertex (1,0) at (1,0) {};
\vertex (-1,0) at (-1,0) {};
\vertex (2,0) at (2,0) {};
\vertex (-2,0) at (-2,0) {};
\vertex (3,0) at (3,0) {};
\vertex (-3,0) at (-3,0) {};

\path
	(0,0) edge (1,0)
	(0,0) edge (-1,0)
	(1,0) edge (2,0)
	(-1, 0) edge (-2, 0)
	(-2, 0) edge (-3,0)
	(2,0) edge (3,0)
	;

\end{tikzpicture}

\end{centering}
\caption{$\mathcal{E}_1(p)$ for $p = 1, 2, 3$}
\end{figure}
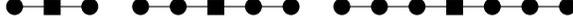

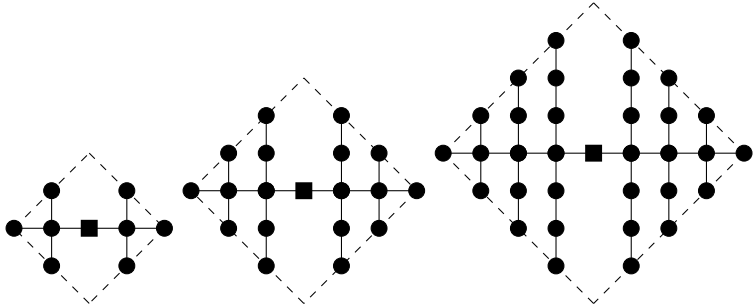
\begin{figure}
\begin{centering}

\begin{tikzpicture}[scale=0.5]

\draw [dashed] (2,0) -- (0,-2);
\draw [dashed] (0,-2) -- (-2,0);
\draw [dashed] (-2,0) -- (0,2);
\draw [dashed] (0,2) -- (2,0);

\vertex (-2,0) at (-2,0) {};
\vertex (-1,0) at (-1,0) {};
\vertex (-1,1) at (-1,1) {};  
\vertex (-1,-1) at (-1,-1) {};

\vertex[shape=rectangle] (0,0) at (0,0) {};

\vertex (2,0) at (2,0) {};
\vertex (1,0) at (1,0) {};
\vertex (1,1) at (1,1) {};  
\vertex (1,-1) at (1,-1) {};

\path
	(-2,0) edge (-1,0)
	
	(-1,0) edge (-1,-1)
	(-1,0) edge (-1,1)
	
	(0,0) edge (-1,0)
	(0,0) edge (1,0)
	
	(1,0) edge (1,-1)
	(1,0) edge (1,1)

	(2,0) edge (1,0)

	 ;   
	 
\end{tikzpicture}
\begin{tikzpicture}[scale=0.5]

\draw [dashed] (3,0) -- (0,-3);
\draw [dashed] (0,-3) -- (-3,0);
\draw [dashed] (-3,0) -- (0,3);
\draw [dashed] (0,3) -- (3,0);

\vertex (-3,0) at (-3,0) {};

\vertex (-2,0) at (-2,0) {};
\vertex (-2,1) at (-2,1) {};  
\vertex (-2,-1) at (-2,-1) {};

\vertex (-1,0) at (-1,0) {};
\vertex (-1,0) at (-1,0) {};
\vertex (-1,1) at (-1,1) {};  
\vertex (-1,-1) at (-1,-1) {};
\vertex (-1,2) at (-1,2) {};  
\vertex (-1,-2) at (-1,-2) {};

\vertex[shape=rectangle]  (0,0) at (0,0) {};

\vertex (1,0) at (1,0) {};
\vertex (1,0) at (1,0) {};
\vertex (1,1) at (1,1) {};  
\vertex (1,-1) at (1,-1) {};
\vertex (1,2) at (1,2) {};  
\vertex (1,-2) at (1,-2) {};

\vertex (2,0) at (2,0) {};
\vertex (2,1) at (2,1) {};  
\vertex (2,-1) at (2,-1) {};

\vertex (3,0) at (3,0) {};

\path
	(-3,0) edge (-2,0)
	
	(-2,0) edge (-2,-1)
	(-2,0) edge (-2,1)
	
	(-2,0) edge (-1,0)
	
	(-1,0) edge (-1,-1)
	(-1,0) edge (-1,1)
	(-1,-1) edge (-1,-2)
	(-1,1) edge (-1,2)
	
	(0,0) edge (-1,0)
	(0,0) edge (1,0)
		
	(1,0) edge (1,-1)
	(1,0) edge (1,1)
	(1,-1) edge (1,-2)
	(1,1) edge (1,2)
	
	(2,0) edge (1,0)

	(2,0) edge (2,-1)
	(2,0) edge (2,1)

	(3,0) edge (2,0)

	 ;   
	 
\end{tikzpicture}
\begin{tikzpicture}[scale=0.5]
\draw [dashed] (4,0) -- (0,-4);
\draw [dashed] (0,-4) -- (-4,0);
\draw [dashed] (-4,0) -- (0,4);
\draw [dashed] (0,4) -- (4,0);

\vertex (-4,0) at (-4,0) {};

\vertex (-3,0) at (-3,0) {};
\vertex (-3,1) at (-3,1) {};  
\vertex (-3,-1) at (-3,-1) {};

\vertex (-2,0) at (-2,0) {};
\vertex (-2,0) at (-2,0) {};
\vertex (-2,1) at (-2,1) {};  
\vertex (-2,-1) at (-2,-1) {};
\vertex (-2,2) at (-2,2) {};  
\vertex (-2,-2) at (-2,-2) {};

\vertex (-1,0) at (-1,0) {};
\vertex (-1,0) at (-1,0) {};
\vertex (-1,1) at (-1,1) {};  
\vertex (-1,-1) at (-1,-1) {};
\vertex (-1,2) at (-1,2) {};  
\vertex (-1,-2) at (-1,-2) {};
\vertex (-1,3) at (-1,3) {};  
\vertex (-1,-3) at (-1,-3) {};

\vertex[shape=rectangle]  (0,0) at (0,0) {};

\vertex (1,0) at (1,0) {};
\vertex (1,0) at (1,0) {};
\vertex (1,1) at (1,1) {};  
\vertex (1,-1) at (1,-1) {};
\vertex (1,2) at (1,2) {};  
\vertex (1,-2) at (1,-2) {};
\vertex (1,3) at (1,3) {};  
\vertex (1,-3) at (1,-3) {};

\vertex (2,0) at (2,0) {};
\vertex (2,0) at (2,0) {};
\vertex (2,1) at (2,1) {};  
\vertex (2,-1) at (2,-1) {};
\vertex (2,2) at (2,2) {};  
\vertex (2,-2) at (2,-2) {};

\vertex (3,0) at (3,0) {};
\vertex (3,1) at (3,1) {};  
\vertex (3,-1) at (3,-1) {};

\vertex (4,0) at (4,0) {};

\path
	(-4,0) edge (-3,0)
	
	(-3,0) edge (-3,-1)
	(-3,0) edge (-3,1)
	(-3,0) edge (-2,0)
	
	(-2,0) edge (-2,-1)
	(-2,0) edge (-2,1)
	(-2,-1) edge (-2,-2)
	(-2,1) edge (-2,2)
	(-2,0) edge (-1,0)
	
	(-1,0) edge (-1, -1)
	(-1,0) edge (-1, 1)
	(-1,1) edge (-1, 2)
	(-1, 2) edge (-1, 3)
	(-1,-1) edge (-1, -2)
	(-1,-2) edge (-1,-3)
	
	(0,0) edge (-1,0)
	(0,0) edge (1,0)
		
	(1,0) edge (1, -1)
	(1,0) edge (1, 1)
	(1,1) edge (1, 2)
	(1, 2) edge (1, 3)
	(1,-1) edge (1, -2)
	(1,-2) edge (1,-3)
	
	(2,0) edge (2,-1)
	(2,0) edge (2,1)
	(2,-1) edge (2,-2)
	(2,1) edge (2,2)
	(2,0) edge (1,0)
	
	(2,0) edge (3,0)

	(3,0) edge (3,-1)
	(3,0) edge (3,1)
	
	(4,0) edge (3,0)

	 ;   

\end{tikzpicture}

\end{centering}
\caption{Construction $\mathcal{E}_2(p)$ for $\D = 4$ and $D = 2,3,4$}

\end{figure}

\begin{proof}
We start by constructing a base case.  In dimension $k=1$, define $\mathcal{E}_1(p)$ to be the induced subgraph on vertices in the interval $[-p,p]$ as shown in Figure 1.  Now assume that we have constructed $\mathcal{E}_{k-1}(p)$ in $\Z^{k-1}$ satisfying the conditions in Proposition \ref{Ekp}.

We construct $\mathcal{E}_k(p)$ from $\mathcal{E}_{k-1}(p)$ as follows: at $x_k = \pm i$, for $1 \leq i \leq p-2$, place a copy of $\mathcal{E}_{k-1}(p-i)$.  We add the vertex $(0,0, \dots, 0) \in \Z^k$ to $\mathcal{E}_k(p)$, and connect $(0,0, \dots, 0, j) \in \Z^k$ with an edge to $(0,0, \dots, 0, j +1)$, for $-(p-2) \leq j < (p-2)$.  Figure 2 shows the construction of $\mathcal{E}_2 (p)$.

First we check that the degrees are bounded by four.  The only vertices whose degree changed were the vertices at $(0,0, \dots, 0, j)$ for $-(p-2) \leq j \leq p-2$.  When $j \neq 0$, this vertex was the center of $\mathcal{E}_{k-1}(p-|j|)$ and thus by condition 2 had degree $2$.  We added $2$ edges to it, making it degree 4.  When $j =0$ this is the origin in $\Z^k$, and we constructed $\mathcal{E}_k(p)$ such that $(0,0, \dots, 0) \in \Z^k$ connects to $(0,0, \dots, 0, \pm 1)$, so it has degree 2.  Thus $\mathcal{E}_k (p)$ satisfies conditions 1 and 2.

Next, we check the condition on the diameter.  Let $v$ be some vertex, not the origin, in $\mathcal{E}_k(p)$.  Assume $v$ is located in the plane $x_k= i$ in $\mathcal{E}_{k-1}(p-|i|)$ for $-(p-2) \leq i \leq p-2$.  Then by condition 3, we know that $v$ is at most $p-|i|$ from $(0,0, \dots, 0, i) \in \Z^k$ and furthermore $(0,0, \dots, 0, i)$ is distance $|i|$ from $(0,0, \dots, 0) \in \Z^k$.  Thus $v$ has distance at most $p$ from $(0,0, \dots, 0)$, and $\mathcal{E}_k(p)$ satisfies condition 3.

Now, we count the number of vertices in $\mathcal{E}_k(p)$.  By induction we know that \\ $|\mathcal{E}_{k-1}(p-i)| = | B^e_{k-1}(p-i)| + O (p^{k-2})$.  Therefore  
\begin{eqnarray*}
|\mathcal{E}_k (p)| &=& 2\sum_{i=1}^{p-2} \ |\mathcal{E}_{k-1}(p-i)| +1\\
 &=& 2 \sum_{i=1}^{p-2}\left( |B^e_{k-1}(p-i)| + O (p^{k-2})\right) +1\\
  &=& 2 \sum_{i=1}^{p-2}\left( \frac{2^{k-1} (p-i)^{k-1}}{(k-1)!} + O (p^{k-2})\right) +1\\
  &=& \frac{2^k p^k}{k!} + O (p^{k-1}).
 \end{eqnarray*} 
 Thus we have checked all four of the conditions, and we can continue on in this manner, constructing $\mathcal{E}_k(p)$ for all $k$.
 
\end{proof}

\begin{prop}

There exists a graph $\mathcal{E}_{k}\p(p)$ centered at the origin satisfying the following conditions:

\begin{enumerate}

\item The degree of any vertex is bounded by 4;

\item The origin, $(0,0, \dots, 0) \in \Z^{k}$ has degree $2$;

\item Any vertex in $\mathcal{E}_{k}\p(p)$ is at most distance $p$ from the vertex $(0,0, \dots, 0) \in \Z^{k}$;

\item $|\mathcal{E}_{k}\p(p)| = \frac{2^{k}p^{k}}{k!} + \frac{2^{k-1}p^{k-1}}{(k-1)!}+ O (p^{k-2}).$ 

\end{enumerate}
\label{Eprime}
\end{prop}
\begin{proof}
We start by constructing a base case.  In dimension $k = 1$, define $\mathcal{E}_1\p(p)$ to be the induced subgraph on vertices in the interval $[-p, p]$.  Now assume that we have constructed $\mathcal{E}_{k-1}\p(p)$ satisfying the conditions in Proposition \ref{Eprime}.

We construct $\mathcal{E}_k\p(p)$ from $\mathcal{E}_{k-1}\p(p)$ and $\mathcal{E}_{k-1}(p)$ as follows: at $x_k = \pm i$, for $2 \leq i \leq p-2$ place a copy of $\mathcal{E}_{k-1}\p(p-i)$.  Also place a copy of $\mathcal{E}_{k-1}\p(p-1)$ at $x_k = -1$.  At $x_k = 1$ place a copy of $\mathcal{E}_{k-1}(p-1)$.  Connect $(0, 0, \dots, 0, j) \in \Z^k$ with an edge to $(0,0, \dots, 0, j+1)$, for $-(p-1) \leq j \leq p-1$, adding the vertex $(0,0, \dots, 0)$ to $\mathcal{E}_k\p(p)$.  Furthermore, we include most of the vertices in $B^e_{k-1}(p-1)$ in the plane $x_k = 0$ by connecting $(v_1, \dots, v_{k-1}, 0)$ to $(v_1, \dots, v_{k-1}, 1)$ if $\sum_{i=1}^{k-1}|v_i| \leq p-2$, $v_{k-1} \leq |p-2|$, and $v_i \neq 0$ for all $i \neq k$.  See Figure 3.  Note that $(v_1, \dots, v_{k-1},0)$ is in $\mathcal{E}_k\p(p)$ because $(v_1, \dots, v_{k-1})$ is in $\mathcal{E}_{k-1}(p-1)$, as $v_i \neq 0$ for all $1 \leq i \leq k-1$.

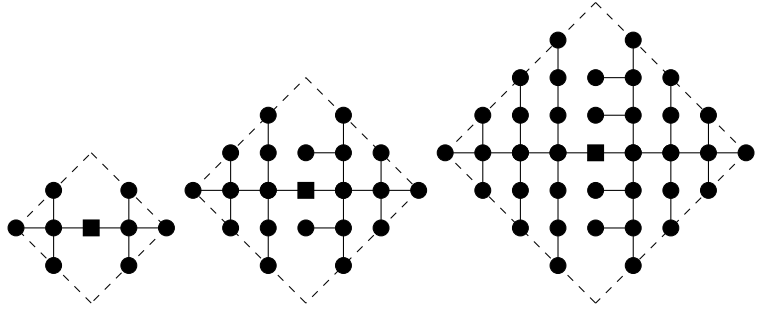
\begin{figure}
\begin{centering}

\begin{tikzpicture}[scale=0.5]

\draw [dashed] (2,0) -- (0,-2);
\draw [dashed] (0,-2) -- (-2,0);
\draw [dashed] (-2,0) -- (0,2);
\draw [dashed] (0,2) -- (2,0);

\vertex (-2,0) at (-2,0) {};
\vertex (-1,0) at (-1,0) {};
\vertex (-1,1) at (-1,1) {};  
\vertex (-1,-1) at (-1,-1) {};

\vertex[shape=rectangle] (0,0) at (0,0) {};

\vertex (2,0) at (2,0) {};
\vertex (1,0) at (1,0) {};
\vertex (1,1) at (1,1) {};  
\vertex (1,-1) at (1,-1) {};

\path
	(-2,0) edge (-1,0)
	
	(-1,0) edge (-1,-1)
	(-1,0) edge (-1,1)
	
	(0,0) edge (-1,0)
	(0,0) edge (1,0)
	
	(1,0) edge (1,-1)
	(1,0) edge (1,1)

	(2,0) edge (1,0)

	 ;   
	 
\end{tikzpicture}
\begin{tikzpicture}[scale=0.5]

\draw [dashed] (3,0) -- (0,-3);
\draw [dashed] (0,-3) -- (-3,0);
\draw [dashed] (-3,0) -- (0,3);
\draw [dashed] (0,3) -- (3,0);

\vertex (-3,0) at (-3,0) {};

\vertex (-2,0) at (-2,0) {};
\vertex (-2,1) at (-2,1) {};  
\vertex (-2,-1) at (-2,-1) {};

\vertex (-1,0) at (-1,0) {};
\vertex (-1,0) at (-1,0) {};
\vertex (-1,1) at (-1,1) {};  
\vertex (-1,-1) at (-1,-1) {};
\vertex (-1,2) at (-1,2) {};  
\vertex (-1,-2) at (-1,-2) {};

\vertex[shape=rectangle]  (0,0) at (0,0) {};
\vertex (0,1) at (0,1) {};
\vertex (0,-1) at (0,-1) {};

\vertex (1,0) at (1,0) {};
\vertex (1,0) at (1,0) {};
\vertex (1,1) at (1,1) {};  
\vertex (1,-1) at (1,-1) {};
\vertex (1,2) at (1,2) {};  
\vertex (1,-2) at (1,-2) {};

\vertex (2,0) at (2,0) {};
\vertex (2,1) at (2,1) {};  
\vertex (2,-1) at (2,-1) {};

\vertex (3,0) at (3,0) {};

\path
	(-3,0) edge (-2,0)
	
	(-2,0) edge (-2,-1)
	(-2,0) edge (-2,1)
	
	(-2,0) edge (-1,0)
	
	(-1,0) edge (-1,-1)
	(-1,0) edge (-1,1)
	(-1,-1) edge (-1,-2)
	(-1,1) edge (-1,2)
	
	(0,0) edge (-1,0)
	(0,0) edge (1,0)
	(0,1) edge (1,1)
	(0,-1) edge (1,-1)

	(1,0) edge (1,-1)
	(1,0) edge (1,1)
	(1,-1) edge (1,-2)
	(1,1) edge (1,2)
	
	(2,0) edge (1,0)

	(2,0) edge (2,-1)
	(2,0) edge (2,1)

	(3,0) edge (2,0)

	 ;   
	 
\end{tikzpicture}
\begin{tikzpicture}[scale=0.5]
\draw [dashed] (4,0) -- (0,-4);
\draw [dashed] (0,-4) -- (-4,0);
\draw [dashed] (-4,0) -- (0,4);
\draw [dashed] (0,4) -- (4,0);

\vertex (-4,0) at (-4,0) {};

\vertex (-3,0) at (-3,0) {};
\vertex (-3,1) at (-3,1) {};  
\vertex (-3,-1) at (-3,-1) {};

\vertex (-2,0) at (-2,0) {};
\vertex (-2,0) at (-2,0) {};
\vertex (-2,1) at (-2,1) {};  
\vertex (-2,-1) at (-2,-1) {};
\vertex (-2,2) at (-2,2) {};  
\vertex (-2,-2) at (-2,-2) {};

\vertex (-1,0) at (-1,0) {};
\vertex (-1,0) at (-1,0) {};
\vertex (-1,1) at (-1,1) {};  
\vertex (-1,-1) at (-1,-1) {};
\vertex (-1,2) at (-1,2) {};  
\vertex (-1,-2) at (-1,-2) {};
\vertex (-1,3) at (-1,3) {};  
\vertex (-1,-3) at (-1,-3) {};

\vertex[shape=rectangle]  (0,0) at (0,0) {};
\vertex (0,1) at (0,1) {};
\vertex (0,2) at (0,2) {};
\vertex (0,-1) at (0,-1) {};
\vertex (0,-2) at (0,-2) {};

\vertex (1,0) at (1,0) {};
\vertex (1,0) at (1,0) {};
\vertex (1,1) at (1,1) {};  
\vertex (1,-1) at (1,-1) {};
\vertex (1,2) at (1,2) {};  
\vertex (1,-2) at (1,-2) {};
\vertex (1,3) at (1,3) {};  
\vertex (1,-3) at (1,-3) {};

\vertex (2,0) at (2,0) {};
\vertex (2,0) at (2,0) {};
\vertex (2,1) at (2,1) {};  
\vertex (2,-1) at (2,-1) {};
\vertex (2,2) at (2,2) {};  
\vertex (2,-2) at (2,-2) {};

\vertex (3,0) at (3,0) {};
\vertex (3,1) at (3,1) {};  
\vertex (3,-1) at (3,-1) {};

\vertex (4,0) at (4,0) {};

\path
	(-4,0) edge (-3,0)
	
	(-3,0) edge (-3,-1)
	(-3,0) edge (-3,1)
	(-3,0) edge (-2,0)
	
	(-2,0) edge (-2,-1)
	(-2,0) edge (-2,1)
	(-2,-1) edge (-2,-2)
	(-2,1) edge (-2,2)
	(-2,0) edge (-1,0)
	
	(-1,0) edge (-1, -1)
	(-1,0) edge (-1, 1)
	(-1,1) edge (-1, 2)
	(-1, 2) edge (-1, 3)
	(-1,-1) edge (-1, -2)
	(-1,-2) edge (-1,-3)
	
	(0,0) edge (-1,0)
	(0,0) edge (1,0)
	(0,1) edge (1,1)
	(0,2) edge (1,2)
	(0,-1) edge (1,-1)
	(0,-2) edge (1,-2)	
		
	(1,0) edge (1, -1)
	(1,0) edge (1, 1)
	(1,1) edge (1, 2)
	(1, 2) edge (1, 3)
	(1,-1) edge (1, -2)
	(1,-2) edge (1,-3)
	
	(2,0) edge (2,-1)
	(2,0) edge (2,1)
	(2,-1) edge (2,-2)
	(2,1) edge (2,2)
	(2,0) edge (1,0)
	
	(2,0) edge (3,0)

	(3,0) edge (3,-1)
	(3,0) edge (3,1)
	
	(4,0) edge (3,0)

	 ;   

\end{tikzpicture}

\end{centering}
\caption{Construction $\mathcal{E}\p_2(p)$ for $\D = 4$ and $D = 2,3,4$}

\end{figure}

Now we consider the maximum degree of $\mathcal{E}_k\p (p)$.  Of the vertices in $\mathcal{E}\p_k(p)$ that are not in the planes $x_k=1$ or $x_k = 0$, the only vertices whose degree increased are $(0,0, \dots, 0, j)$ for $-(p-1) \leq j \leq p-1$.  Because those were the centers of $\mathcal{E}\p_{k-1}(p-i)$ for $1 \leq i \leq p-2$ they had degree 2 by condition 2 of Proposition \ref{Eprime}, and thus now have degree 4.  The vertices in $\mathcal{E}\p_k(p)$ in the plane $x_k= 0$ all have degree 1.  Finally, the vertices in the plane $x_k = 1$ were in the graph $\mathcal{E}_{k-1} (p-1)$, and the only ones that had their degree increased were those such that $x_i \neq 0$ for all $1 \leq i \leq k$.  Therefore by condition 1 of Proposition \ref{Ekp}, they were originally at most degree 2 and are now at most degree 3.  Thus the maximum degree of $\mathcal{E}\p_k (p)$ is 4, satisfying condition 1 of Proposition \ref{Eprime}.  Note also that the origin only connects to two other nodes, $(0,0, \dots, 0, -1,)$ and $(0,0, \dots, 0, 1)$ and so has degree 2, satisfying condition 2 of Proposition \ref{Eprime}.

Next, we check the condition on the diameter.  Let $v$ be some vertex, not the origin, in $\mathcal{E}\p_k(p)$.  Assume $v$ is located in the plane $x_k= i$ in $\mathcal{E}\p_{k-1}(p-|i|)$ for $-(p-2) \leq i \leq p-2$ when $i \neq 0, 1$.  Then by condition 4 in Proposition \ref{Eprime}, we know that $v$ is at most distance $p-|i|$ from $(0,0, \dots, 0, i) \in \Z^k$ and furthermore $(0,0, \dots, 0, i)$ is distance $|i|$ from $(0,0, \dots, 0) \in \Z^k$.  Thus $v$ has distance at most $p$ from $(0,0, \dots, 0)$, and $\mathcal{E}\p_k(p)$ satisfies condition 4 of Proposition \ref{Eprime}. If $v = (v_1, \dots, v_{k-1},0)$ is in the plane $x_k = 0$, then note that $v_{k-1} \leq |p-2|$, and it is distance 1 from $(v_1, \dots, v_{k-1}, 1)$ which is distance at most $p-2$ from $(0,0, \dots, 0, 1)$ which is distance 1 from the origin, so $v$ is at most distance $p$ from the origin.  If $v$ is in the plane $x_k = 1$ then by condition 3 of Proposition \ref{Ekp}, $v$ is at most distance $p -1$ from the point $(0,0, \dots, 1)$ which is distance $1$ from the origin.  Thus $\mathcal{E}\p_k(p)$ satisfies condition 3 of Proposition \ref{Eprime}. 

Finally, we count the number of vertices in $\mathcal{E}\p_k(p)$.  By condition 4 of Proposition \ref{Eprime}, we know that $|\mathcal{E}\p_{k-1}(p-i)| = |B^e_{k-1}(p-i)| + O (p^{k-3})$ and by condition 4 of Proposition \ref{Ekp} we know that $| \mathcal{E}_{k-1}(p-1)| = |B^e_{k-1}(p-1)| + O (p^{k-3})$.  We also included all of the vertices $(v_1, \dots, v_{k})$ in $B^e_{k-1}(p)$ in the plane $x_k=0$ except those such that $\sum_{i=1}^{k-1} v_i = p-1$, $v_{k-1} = p-1$ or $p$, and those with $v_i = 0$ for some $ 1 \leq i \leq k-1$.  These are all sets of size $O(p^{k-2})$.  Therefore 

\begin{eqnarray*} 
| \mathcal{E}\p_k(p)| &=& 2 \left( \sum_{i=2}^{p-2} | \mathcal{E}\p_{k-1}(p-i) | \right) + |\mathcal{E}\p_{k-1} (p-1)| + | \mathcal{E}_{k-1}(p-1)| + O (p^{k-2}) \\
 &=& \sum_{i=2}^{p-2} \left( |B^e_{k-1}(p-i)| + O(p^{k-3}) \right) + | B^e_{k-1}(p-1)| + O (p^{k-3}) + | B^e _{k-1}(p-1)| + O (p^{k-2}) \\
 &=& 2 \sum_{i=1}^{p-2} \left( |B^e_{k-1} (p-i)| + O (p^{k-3}) \right) + O (p^{k-2}) \\
&=& 2 \sum_{i=1}^{p-2} \left( \frac{2^{k-1}(p-i)^{k-1}}{(k-1)!} + \frac{2^{k-2}(p-i)^{k-2}}{(k-2)!} + O (p^{k-3}) \right) + O (p^{k-2}) \\
&=& \frac{2^kp^k}{k!}+\frac{2^{k-1}p^{k-1}}{(k-1)!} + O (p^{k-2}).
\end{eqnarray*}

Thus we have checked all four of the conditions, and we can continue in the manner, giving us the lower bound for $N^e_k(\D, p)$ when $\D = 4$ stated in Theorem 1 Part 4.
\end{proof}

\subsection{The Odd Diameter Case}

A similar pair of constructions gives us the lower bounds for the odd diameter case $N_k^o(\D, p)$ when $\D =4$, stated in Theorem 2 Part 4.

\begin{prop} There exists a graph $\mathcal{O}_{k}(p)$ in $ (\Z+\frac{1}{2}) \times  \Z^{k-1}$ satisfying the following conditions:

\begin{enumerate}

\item The degree of any vertex $v= (v_1, \dots, v_{k})$ which is not at $(\pm 1/2,0, \dots,0 ) \in  (\Z +\frac{1}{2}) \times \Z^{k-1}$ is at most 4 if $v_1 = \pm 1/2$ and less than or equal to $2$ otherwise;

\item The vertices at $(\pm 1/2,0, \dots, 0) \in (\Z +\frac{1}{2}) \times \Z^{k-1}$ have degree $2$;

\item Any vertex in $\mathcal{O}_{k}(p)$ is at most distance $p$ from one of the two vertices $(\pm 1/2,0, \dots, 0) \in (\Z + \frac{1}{2} ) \times \Z^{k-1}$, and at most distance $p+1$ from the other;

\item $|\mathcal{O}_{k}(p)| = \frac{2^{k} p^{k}}{(k)!}   + O (p^{k-1})$.

\end{enumerate}
\label{Okp}
\end{prop}

\begin{proof}
We start by constructing a base case.  Recall that in the odd case we consider $G^o$ with vertices in $(\Z + \frac{1}{2}) \times \Z^{k-1}$ where vertices have fractional coordinates in the first dimension and integer coordinates in all others.  In dimension $k=1$, define $\mathcal{O}_1(p)$ to be the induced subgraph given by vertices in the interval $[-p-1/2,p+1/2]$ as shown in Figure 4.  Now assume that we have constructed $\mathcal{O}_{k-1}(p)$ satisfying the conditions in Proposition \ref{Okp}.

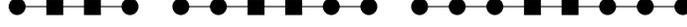
\begin{figure}[h]
\begin{centering}

\begin{tikzpicture}[scale=0.5]

\vertex[shape=rectangle] (.5,0) at (.5,0) {};
\vertex[shape=rectangle] (-.5,0) at (-.5,0) {};
\vertex (1.5,0) at (1.5,0) {};
\vertex (-1.5,0) at (-1.5,0) {};

\path
	(.5,0) edge (-.5,0)
	(-1.5,0) edge (-.5,0)
	(1.5,0) edge (.5,0)

	;

\end{tikzpicture} \,\,\,
\begin{tikzpicture}[scale=0.5]

\vertex[shape=rectangle] (.5,0) at (.5,0) {};
\vertex[shape=rectangle] (-.5,0) at (-.5,0) {};
\vertex (1.5,0) at (1.5,0) {};
\vertex (-1.5,0) at (-1.5,0) {};
\vertex (2.5,0) at (2.5,0) {};
\vertex (-2.5,0) at (-2.5,0) {};

\path
	(.5,0) edge (-.5,0)
	(-1.5,0) edge (-.5,0)
	(1.5,0) edge (.5,0)
	(-2.5,0) edge (-1.5,0)
	(2.5,0) edge (1.5,0)
	;

\end{tikzpicture} \,\,\,
\begin{tikzpicture}[scale=0.5]

\vertex[shape=rectangle] (.5,0) at (.5,0) {};
\vertex[shape=rectangle] (-.5,0) at (-.5,0) {};
\vertex (1.5,0) at (1.5,0) {};
\vertex (-1.5,0) at (-1.5,0) {};
\vertex (2.5,0) at (2.5,0) {};
\vertex (-2.5,0) at (-2.5,0) {};
\vertex (3.5,0) at (3.5,0) {};
\vertex (-3.5,0) at (-3.5,0) {};

\path
	(.5,0) edge (-.5,0)
	(-1.5,0) edge (-.5,0)
	(1.5,0) edge (.5,0)
	(-2.5,0) edge (-1.5,0)
	(2.5,0) edge (1.5,0)
	(-3.5,0) edge (-2.5,0)
	(3.5,0) edge (2.5,0)
	;

\end{tikzpicture}

\end{centering}
\caption{$\mathcal{O}_1(p)$ for $p = 1, 2, 3$}
\end{figure}

\begin{figure}
\begin{centering}

\begin{tikzpicture}[scale=.5]
\draw [dashed] (2.5,0) -- (0,-2.5);
\draw [dashed] (0,-2.5) -- (-2.5,0);
\draw [dashed] (-2.5,0) -- (0,2.5);
\draw [dashed] (0,2.5) -- (2.5,0);

\vertex (-2, .5) at (-2, .5) {};
\vertex (-2, -.5) at (-2, -.5) {};
\vertex (-1, 1.5) at (-1, 1.5){};
\vertex (-1, .5) at (-1, .5){};
\vertex (-1, -1.5) at (-1, -1.5){};
\vertex (-1, -.5) at (-1, -.5){};
\vertex[shape=rectangle] (0, .5) at (0, .5){};
\vertex[shape=rectangle] (0, -.5) at (0, -.5){};
\vertex (1, 1.5) at (1, 1.5){};
\vertex (1, .5) at (1, .5){};
\vertex (1, -1.5) at (1, -1.5){};
\vertex (1, -.5) at (1, -.5){};
\vertex (2, .5) at (2, .5){};
\vertex (2, -.5) at (2, -.5){};

\path
	(-2,.5) edge (-1, .5)
	(-2, -.5) edge (-1, -.5)
	(-1, .5) edge (-1, -.5)
	(-1, .5) edge (-1, 1.5)
	(-1, -.5) edge (-1, -1.5)
	(-1,.5) edge (0, .5)
	(-1, -.5) edge (0, -.5)
	(1,.5) edge (0, .5)
	(1, -.5) edge (0, -.5)
	(2,.5) edge (1, .5)
	(1, .5) edge (1, -.5)
	(1, .5) edge (1, 1.5)
	(1, -.5) edge (1, -1.5)
	(2, -.5) edge (1, -.5)

	;

\end{tikzpicture}
\begin{tikzpicture}[scale=.5]
\draw [dashed] (3.5,0) -- (0,-3.5);
\draw [dashed] (0,-3.5) -- (-3.5,0);
\draw [dashed] (-3.5,0) -- (0,3.5);
\draw [dashed] (0,3.5) -- (3.5,0);

\vertex (-3, .5) at (-3, .5) {};
\vertex (-3, -.5) at (-3, -.5) {};

\vertex (-2, 1.5) at (-2, 1.5){};
\vertex (-2, .5) at (-2, .5){};
\vertex (-2, -1.5) at (-2, -1.5){};
\vertex (-2, -.5) at (-2, -.5){};

\vertex (-1, 1.5) at (-1, 1.5){};
\vertex (-1, .5) at (-1, .5){};
\vertex (-1, -1.5) at (-1, -1.5){};
\vertex (-1, -.5) at (-1, -.5){};
\vertex (-1, 2.5) at (-1, 2.5){};
\vertex (-1, -2.5) at (-1, -2.5){};

\vertex[shape=rectangle] (0, .5) at (0, .5){};
\vertex[shape=rectangle] (0, -.5) at (0, -.5){};

\vertex (1, 1.5) at (1, 1.5){};
\vertex (1, .5) at (1, .5){};
\vertex (1, -1.5) at (1, -1.5){};
\vertex (1, -.5) at (1, -.5){};
\vertex (1, 2.5) at (1, 2.5){};
\vertex (1, -2.5) at (1, -2.5){};

\vertex (2, 1.5) at (2, 1.5){};
\vertex (2, .5) at (2, .5){};
\vertex (2, -1.5) at (2, -1.5){};
\vertex (2, -.5) at (2, -.5){};

\vertex (3, .5) at (3, .5){};
\vertex (3, -.5) at (3, -.5){};

\path

	(-3, .5) edge (-2, .5)
	(-3, -.5) edge (-2, -.5)

	(-2, .5) edge (-2, -.5)
	(-2,.5) edge (-1, .5)
	(-2, -.5) edge (-1, -.5)
	(-2, 1.5) edge (-2, .5)
	(-2, -1.5) edge (-2, -.5)

	(-1, .5) edge (-1, -.5)
	(-1, .5) edge (-1, 1.5)
	(-1, -.5) edge (-1, -1.5)
	(-1,.5) edge (0, .5)
	(-1, -.5) edge (0, -.5)
	(-1, -2.5) edge (-1, -1.5)
	(-1, 2.5) edge (-1, 1.5)
	
	(1,.5) edge (0, .5)
	(1, -.5) edge (0, -.5)

	(1, .5) edge (1, -.5)
	(1, .5) edge (1, 1.5)
	(1, -.5) edge (1, -1.5)
	(1, -2.5) edge (1, -1.5)
	(1, 2.5) edge (1, 1.5)

	(2, .5) edge (2, -.5)
	(2,.5) edge (1, .5)
	(2, -.5) edge (2, .5)
	(2, -.5) edge (1, -.5)
	(2, 1.5) edge (2, .5)
	(2, -1.5) edge (2, -.5)
	
	(3, .5) edge (2, .5)
	(3, -.5) edge (2, -.5)

	;

\end{tikzpicture}

\caption{Construction $\mathcal{O}_2(p)$ for $\D = 4$ and $p=2,3$}
\end{centering}
\end{figure}
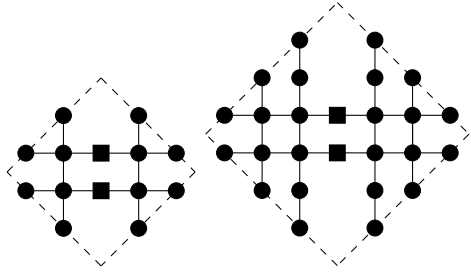

We construct $\mathcal{O}_k(p)$ from $\mathcal{O}_{k-1}(p)$ as follows: at $x_k = \pm i$, for $1 \leq i \leq p-2$, place a copy of $\mathcal{O}_{k-1}(p-i)$.  Connect $(1/2,0, \dots 0, j) \in \Z^k$ with an edge to $(1/2,0, \dots, 0, j +1)$ when $-(p-2) \leq j < (p-2)$, adding the vertex $(1/2,0, \dots ,0) \in \Z^k$ to $\mathcal{O}_k(p)$, and similarly, connect $(-1/2,0, \dots, 0, j) \in \Z^k$ with an edge to $(-1/2,0, \dots, 0, j +1)$ when $-(p-2) \leq j < (p-2)$, adding the vertex $(-1/2,0, \dots, 0,0) \in \Z^k$ to $\mathcal{O}_k(p)$.  Figure 5 shows the construction of $\mathcal{O}_2(p)$.

Then the only vertices which we added edges to were the vertices at $(\pm 1/2,0, \dots, 0, j)$ for $-(p-2) \leq j \leq p-2$.  When $j \neq 0$, by condition 2 of Proposition \ref{Okp} this vertex had degree $2$, and we added $2$ edges, leaving the degree bounded by 4.  When $j =0$, we constructed $\mathcal{O}_k(p)$ such that $(1/2,0, \dots, 0) \in \Z^k$ connects to $(1/2,0, \dots, 0, \pm 1)$, and  $(-1/2, 0, \dots, 0, \pm 1)$  so that it connects to $(-1/2, 0, \dots, 0, \pm 1)$ so they each have degree 2, satisfying condition 2 in Proposition \ref{Okp}.

Now we check the conditions on the diameter.  Note that in order to show $\mathcal{O}_k(p)$ has diameter $2p+1$, it suffices to check condition 3 in Proposition \ref{Okp} that each vertex in $\mathcal{O}_k(p)$ is at most distance $p$ from one of $(\pm 1/2,0, \dots, 0)$ and $p+1$ from the other, because any two vertices in $\mathcal{O}_k(p)$ are either both distance at most $p$ from one of $(\pm 1/2,0, \dots, 0)$, and thus distance at most $2p$ from each other, or for either $(\pm 1/2,0, \dots, 0)$, one vertex is distance at most $p$ away and the other is at most $p+1$ away, and the two vertices are distance at most $2p+1$ from each other.

For $k >1$, the vertices at $(\pm 1/2,0, \dots, 0)$ are distance $3$ apart.  Let $v$ be any vertex that is not at $(\pm 1/2,0, \dots, 0)$ in $\mathcal{O}_k(p)$.  Say $v$ is located in the plane $x_k= i$ in a copy of $\mathcal{O}_k(p-i)$ for $-(p-2) \leq i \leq p-2$.  Then by condition 3 of Proposition \ref{Okp}, we know that $v$ is at most distance $p-|i|$ from one of the vertices at $(\pm 1/2,0, \dots, 0)$, without loss of generality assume it is, $(1/2,0,\dots, 0, i)$ and $p+1-|i|$ from the other, $(-1/2,0,\dots, 0, i)$.  Then $(1/2,0, \dots, 0, i)$ is distance $|i|$ from $(1/2,0,\dots, 0, 0)$, and $(-1/2,0, \dots, 0, i)$  is distance $|i| $ from $(1/2,0,\dots, 0, 0)$, showing that $v$ is at most distance $p$ from $(1/2,0, \dots, 0)$ and distance $p+1$ from $(-1/2, 0, \dots, 0)$.

Finally, we count the number of vertices in $\mathcal{O}_k(p)$.  By condition 4 of Proposition \ref{Okp}, we know that $|\mathcal{O}_{k-1}(p-i)| = | B^o_{k-1}(p-i)| + O (p^{k-2})$.  Therefore

\begin{eqnarray*}
|\mathcal{O}_k (p)| &=& 2\sum_{i=1}^{p-2}|\mathcal{O}_{k-1}(p-i)| +2 \\
&=& 2 \sum_{i=1}^{p-2}\left( |B^o_{k-1}(p-i)| + O (p^{k-2})\right) +2\\
&=& 2 \sum_{i=1}^{p-2}\left( \frac{2^{k-1} (p-i)^{k-1}}{(k-1)!} + O (p^{k-2})\right) +2\\
&=& \frac{2^k p^k}{k!} + O (p^{k-1}).
\end{eqnarray*}
 Thus all four of the conditions in Proposition \ref{Okp} are satisfied, and we can construct $\mathcal{O}_k(p)$ for all $k$.
\end{proof}

\begin{prop}

There exists a graph $\mathcal{O}_{k}\p(p)$ centered at the origin satisfying the following conditions:

\begin{enumerate}

\item The degree of any vertex is bounded by 4;

\item The vertices at $(\pm 1/2,0, \dots, 0) \in (\Z + \frac{1}{2}) \times \Z^{k-1}$ have degree $2$;

\item Any vertex in $\mathcal{O}_{k}\p(p)$ is at most distance $p$ from one of the vertices $(\pm 1/2,0, \dots, 0) \in (\Z + \frac{1}{2}) \times \Z^{k-1}$ and at most distance $p+1$ from the other;

\item $|\mathcal{O}_{k}\p(p)| = \frac{2^{k}p^{k}}{k!} + \frac{2^{k}p^{k-1}}{(k-1)!}+ O (p^{k-2}).$ 

\end{enumerate}
\label{Oprime}
\end{prop}
\begin{proof}
We start by constructing a base case.  In dimension $k=1$, define $\mathcal{O}\p_1(p)$ to be the induced subgraph given by vertices in the interval $[-p-1/2,p+1/2]$ as shown in Figure 4.  Now assume that we have constructed $\mathcal{O}\p_{k-1}(p)$ satisfying the conditions in Proposition \ref{Oprime}.

We construct $\mathcal{O}\p_k(p)$ from $\mathcal{O}\p_{k-1}(p)$ and $\mathcal{O}_{k-1}(p)$ as follows: at $x_k = \pm i$, for $1 \leq i \leq p-2$, $i \neq 1$, place a copy of $\mathcal{O}\p_{k-1}(p-i)$.  At $i = 1$ place a copy of $\mathcal{O}_{k-1}(p-1)$.  Connect $(1/2,0, \dots ,0, j) \in \Z^k$ with an edge to $(1/2,0, \dots, 0, j +1)$ when $-(p-2) \leq j < (p-2)$, adding the vertex $(1/2,0, \dots ,0) \in \Z^k$ to $\mathcal{O}_k(p)$, and similarly, connect $(-1/2,0, \dots, 0, j) \in \Z^k$ with an edge to $(-1/2,0, \dots, 0, j +1)$ when $-(p-2) \leq j < (p-2)$, adding the vertex $(-1/2,0, \dots, 0,0) \in \Z^k$ to $\mathcal{O}_k(p)$.  We also include most of the vertices in $B^o_{k-1}(p-1)$ in the plane $x_k =0$.  Let $v= (v_1, \dots , v_k)$ a vertex in the plane $x_k = 0$.  Connect $v$ to $(v_1, \dots, 1)$ if $\sum_{j=1}^k |v_j| \leq p-2$, $|v_{k-1}| \leq p -3/2$, and $v_1 \neq \pm 1/2$.   See Figure 6.

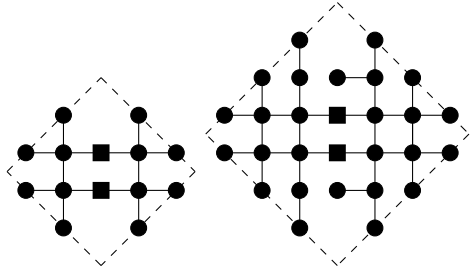
\begin{figure}
\begin{centering}

\begin{tikzpicture}[scale=.5]
\draw [dashed] (2.5,0) -- (0,-2.5);
\draw [dashed] (0,-2.5) -- (-2.5,0);
\draw [dashed] (-2.5,0) -- (0,2.5);
\draw [dashed] (0,2.5) -- (2.5,0);

\vertex (-2, .5) at (-2, .5) {};
\vertex (-2, -.5) at (-2, -.5) {};
\vertex (-1, 1.5) at (-1, 1.5){};
\vertex (-1, .5) at (-1, .5){};
\vertex (-1, -1.5) at (-1, -1.5){};
\vertex (-1, -.5) at (-1, -.5){};
\vertex[shape=rectangle] (0, .5) at (0, .5){};
\vertex[shape=rectangle] (0, -.5) at (0, -.5){};
\vertex (1, 1.5) at (1, 1.5){};
\vertex (1, .5) at (1, .5){};
\vertex (1, -1.5) at (1, -1.5){};
\vertex (1, -.5) at (1, -.5){};
\vertex (2, .5) at (2, .5){};
\vertex (2, -.5) at (2, -.5){};

\path
	(-2,.5) edge (-1, .5)
	(-2, -.5) edge (-1, -.5)
	(-1, .5) edge (-1, -.5)
	(-1, .5) edge (-1, 1.5)
	(-1, -.5) edge (-1, -1.5)
	(-1,.5) edge (0, .5)
	(-1, -.5) edge (0, -.5)
	(1,.5) edge (0, .5)
	(1, -.5) edge (0, -.5)
	(2,.5) edge (1, .5)
	(1, .5) edge (1, -.5)
	(1, .5) edge (1, 1.5)
	(1, -.5) edge (1, -1.5)
	(2, -.5) edge (1, -.5)

	;

\end{tikzpicture}
\begin{tikzpicture}[scale=.5]
\draw [dashed] (3.5,0) -- (0,-3.5);
\draw [dashed] (0,-3.5) -- (-3.5,0);
\draw [dashed] (-3.5,0) -- (0,3.5);
\draw [dashed] (0,3.5) -- (3.5,0);

\vertex (-3, .5) at (-3, .5) {};
\vertex (-3, -.5) at (-3, -.5) {};

\vertex (-2, 1.5) at (-2, 1.5){};
\vertex (-2, .5) at (-2, .5){};
\vertex (-2, -1.5) at (-2, -1.5){};
\vertex (-2, -.5) at (-2, -.5){};

\vertex (-1, 1.5) at (-1, 1.5){};
\vertex (-1, .5) at (-1, .5){};
\vertex (-1, -1.5) at (-1, -1.5){};
\vertex (-1, -.5) at (-1, -.5){};
\vertex (-1, 2.5) at (-1, 2.5){};
\vertex (-1, -2.5) at (-1, -2.5){};

\vertex[shape=rectangle] (0, .5) at (0, .5){};
\vertex[shape=rectangle] (0, -.5) at (0, -.5){};
\vertex (0, 1.5) at (0, 1.5) {};
\vertex (0, -1.5) at (0, -1.5) {};

\vertex (1, 1.5) at (1, 1.5){};
\vertex (1, .5) at (1, .5){};
\vertex (1, -1.5) at (1, -1.5){};
\vertex (1, -.5) at (1, -.5){};
\vertex (1, 2.5) at (1, 2.5){};
\vertex (1, -2.5) at (1, -2.5){};

\vertex (2, 1.5) at (2, 1.5){};
\vertex (2, .5) at (2, .5){};
\vertex (2, -1.5) at (2, -1.5){};
\vertex (2, -.5) at (2, -.5){};

\vertex (3, .5) at (3, .5){};
\vertex (3, -.5) at (3, -.5){};

\path

	(-3, .5) edge (-2, .5)
	(-3, -.5) edge (-2, -.5)

	(-2, .5) edge (-2, -.5)
	(-2,.5) edge (-1, .5)
	(-2, -.5) edge (-1, -.5)
	(-2, 1.5) edge (-2, .5)
	(-2, -1.5) edge (-2, -.5)

	(-1, .5) edge (-1, -.5)
	(-1, .5) edge (-1, 1.5)
	(-1, -.5) edge (-1, -1.5)
	(-1,.5) edge (0, .5)
	(-1, -.5) edge (0, -.5)
	(-1, -2.5) edge (-1, -1.5)
	(-1, 2.5) edge (-1, 1.5)
	
	(1,.5) edge (0, .5)
	(1, -.5) edge (0, -.5)
	(1, 1.5) edge (0, 1.5)
	(1, -1.5) edge (0, -1.5)

	(1, .5) edge (1, -.5)
	(1, .5) edge (1, 1.5)
	(1, -.5) edge (1, -1.5)
	(1, -2.5) edge (1, -1.5)
	(1, 2.5) edge (1, 1.5)

	(2, .5) edge (2, -.5)
	(2,.5) edge (1, .5)
	(2, -.5) edge (2, .5)
	(2, -.5) edge (1, -.5)
	(2, 1.5) edge (2, .5)
	(2, -1.5) edge (2, -.5)
	
	(3, .5) edge (2, .5)
	(3, -.5) edge (2, -.5)

	;

\end{tikzpicture}

\caption{Construction $\mathcal{O}\p_2(p)$ for $\D = 4$ and $p=2,3$}
\end{centering}
\end{figure}

Next we consider the maximum degree of $\mathcal{O}_k \p (p)$.  Of the vertices in $\mathcal{O}\p_k (p)$ that are not in the planes $x_k =1$ or $x_k = 0$, the only vertices whose degree increased are the vertices at $(\pm 1/2, 0, \dots, 0, j)$ for $-(p-2) \leq j \leq p-2$.  By condition 2 of Proposition \ref{Oprime} these vertices had degree $2$, and we added $2$ edges, leaving the degree bounded by 4.  The vertices in the plane $x_k = 0$ all have degree 1.  Finally, the vertices in the plane $x_k = 1$ were in the graph $\mathcal{O}_{k-1}(p-1)$ and the only ones that had their degree increased were those such that $x_k \neq \pm 1/2$.  Therefore by condition 1 of Proposition \ref{Okp}, they were originally degree 2 and are now degree 3.  Thus the maximum degree of $\mathcal{O}\p_k (p)$ is 4, satisfying condition 1 of Proposition \ref{Oprime}.  Note also that the vertices at $(\pm 1/2,0, \dots, 0) $ connect to only two other nodes and so have degree 2, satisfying condition 2 of Proposition \ref{Oprime}. 

Now we check the condition on the diameter.  Let $v$ be some vertex, not the origin, in $\mathcal{O}\p_k (p)$.  Assume $v$ is located in the plane $x_k = i$ in a copy of $\mathcal{O}\p_{k-1}(p-|i|)$ for $-(p-2) \leq i \leq p-2$ when $i \neq 0, 1$.  Then by condition 3 of Proposition \ref{Oprime}, we know that $v$ is at most distance $p- |i|$ from one of the vertices at $(\pm 1/2,0, \dots, 0)$, without loss of generality assume it is, $(1/2,0,\dots, 0, i)$ and $p+1-|i|$ from the other, $(-1/2,0,\dots, 0, i)$.  Then $(1/2,0, \dots, 0, i)$ is distance $|i|$ from $(1/2,0,\dots, 0)$, and $(-1/2,0, \dots, 0, i)$  is distance $|i| +1$ from $(1/2,0,\dots, 0)$, showing that $v$ is at most distance $p$ from $(1/2,0, \dots, 0)$ and distance $p+1$ from $(-1/2, 0, \dots, 0)$.  If $v$ is in the plane $x_k = 1$ then it is in a copy of $\mathcal{O}_{k-1}(p-1)$ and by condition 3 of Proposition \ref{Okp} without loss of generality it is distance at most $p-1$ from $(1/2, 0, \dots, 0, 1)$ and distance at most $p$ from $(-1/2, 0, \dots, 0, 1)$.  Therefore it is at most distance $p$ from $(1/2, 0, \dots, 0,0)$ and at most distance $p+1$ from $(-1/2, 0, \dots, 0, 0)$. If $v= (v_1, \dots, v_k)$ is in the plane $x_k = 0$ then note that $|v_{k-1}| \leq p - 3/2$.  We know $v$ is distance $1$ from $(v_1, \dots, v_{k-1}, 1)$ which, without loss of generality, is distance at most $p - 1$ from $(1/2, 0, \dots , 0, 1)$ and $p$ from $(-1/2, 0, \dots, 0, 1)$.  Therefore $v$ is at most distance $p$ from $(1/2, 0, \dots, 0)$ and $p+1$ from $(-1/2, 0, \dots, 0)$.  Thus $\mathcal{O}\p_k (p)$ satisfies condition 3 of Proposition \ref{Oprime}.

Finally, we count the number of vertices in $\mathcal{O}\p_k(p)$.  By condition 4 of Proposition \ref{Oprime}, we know that $|\mathcal{O}\p_{k-1}(p-i)| = |B^o_{k-1}(p-i)| + O (p^{k-3})$ and by condition 4 of Proposition \ref{Okp} we know that $| \mathcal{O}_{k-1}(p-1) = |B^o_{k-1}(p-1)| + O (p^{k-3})$.  We also included all of the vertices $(v_1, \dots, v_k)$ in $B^o_{k-1}(p)$ in the plane $x_k=0$ except those such that $v_{k-1} = p-1$ or $p$ (or when $k = 2,$ we excluded those such that $v_{k-1} = p+ 1/2$ or $v_{k-1} = p-1/2$), $\sum_{i=1}^k v_i = p-1$, and those with $v_1 = \pm 1/2$.  These are all sets of size $O(p^{k-2})$.  Therefore 

\begin{eqnarray*} 
| \mathcal{O}\p_k(p)| &=& 2 \left( \sum_{i=2}^{p-2} | \mathcal{O}\p_{k-1}(p-i) | \right) + |\mathcal{O}\p_{k-1} (p-1)| + | \mathcal{O}_{k-1}(p-1)| + O (p^{k-2})\\
 &=& \sum_{i=2}^{p-2} \left( |B^o_{k-1}(p-i)| + O(p^{k-3}) \right) + | B^o_{k-1}(p-1)| + O (p^{k-3}) + | B^o _{k-1}(p-1)| + O (p^{k-2}) \\
 &=& 2 \sum_{i=1}^{p-2} \left( |B^o_{k-1} (p-i)| + O (p^{k-3}) \right) + O (p^{k-2}) \\
&=& 2 \sum_{i=1}^{p-2} \left( \frac{2^{k-1}(p-i)^{k-1}}{(k-1)!} + \frac{2^{k-1}(p-i)^{k-2}}{(k-2)!} + O (p^{k-3}) \right) + O (p^{k-2}) \\
&=& \frac{2^kp^k}{k!}+\frac{2^{k}p^{k-1}}{(k-1)!} + O (p^{k-2}).
\end{eqnarray*}

Thus we have checked all four of the conditions, and we can continue in the manner, giving us the lower bound for $N^o_k(\D, p)$ when $\D = 4$ stated in Theorem 2.
\end{proof}

\section{Bounds when $\D =1,2,3$}

The above constructions cover all $\D \geq 4$.  Now we look at the small cases and prove the bounds in part three of Theorem 1 and Theorem 2.

\begin{lemma}

When $\D = 1$, $N_k^e(\D, p) = N_k^o(\D, p) = 2$.

\end{lemma}
\begin{proof}
The only connected graphs with degree at most $1$ are the single vertex and a pair of vertices joined by an edge, both of which are subgraphs of the mesh, so the maximum size of a connected subgraph of degree bounded by 1 and diameter bounded by $D$ is just 2.

\end{proof}

\begin{lemma}

 When $\D=2$, $N^e_k (\D, p) = 4p$ and  $N^o_k(\D, p) = 4p+2$.

\end{lemma}
\begin{proof}
If $\D = 2$ then the only connected subgraphs of the mesh satisfying this bound on degree are (not necessarily straight) lines or even cycles.  Any line of length $D$ has diameter $D$, and any cycle of length $2D$ has diameter $D$, so $N^e_k(2, p) = 4p$ and $N^o_k(2, p)= 4p+2$.
\end{proof}

\begin{figure}
\begin{centering}
\begin{tikzpicture}[scale = .5]

\vertex (-2, .5) at (-2, .5) {};
\vertex (-2, -.5) at (-2, -.5) {};
\vertex (-2, 1.5) at (-2, 1.5){};
\vertex(-2, -1.5) at (-2, -1.5){};
\vertex (-2, 2.5) at (-2, 2.5){};

\vertex (-1, 1.5) at (-1, 1.5){};
\vertex (-1, .5) at (-1, .5){};
\vertex (-1, -1.5) at (-1, -1.5){};
\vertex (-1, -.5) at (-1, -.5){};
\vertex (-1, 2.5) at (-1, 2.5){};

\vertex (0, .5) at (0, .5){};
\vertex (0, -.5) at (0, -.5){};
\vertex (0, 1.5) at (0, 1.5){};
\vertex(0, -1.5) at (0, -1.5){};
\vertex (0, 2.5) at (0, 2.5){};

\vertex (1, 1.5) at (1, 1.5){};
\vertex (1, .5) at (1, .5){};
\vertex (1, -1.5) at (1, -1.5){};
\vertex (1, -.5) at (1, -.5){};
\vertex (1, 2.5) at (1, 2.5){};

\vertex (2, .5) at (2, .5){};
\vertex (2, -.5) at (2, -.5){};
\vertex (2, 1.5) at (2, 1.5){};
\vertex(2, -1.5) at (2, -1.5){};
\vertex (2, 2.5) at (2, 2.5){};

\path
	(-2, -0.5) edge (-1, -0.5)
	(-1, -1.5) edge (0, -1.5)
	(0, -0.5) edge (1, -0.5)
	(1, -1.5) edge (2, -1.5)
	
	(1, 2.5) edge (1, 1.5)
	(1, 1.5) edge (1, .5)
	(1, .5) edge (1, -.5)
	(1, -.5) edge (1, -1.5)
	
	(2, 2.5) edge (2, 1.5)
	(2, 1.5) edge (2, .5)
	(2, .5) edge (2, -.5)
	(2, -.5) edge (2, -1.5)
	
	(-1, 2.5) edge (-1, 1.5)
	(-1, 1.5) edge (-1, .5)
	(-1, .5) edge (-1, -.5)
	(-1, -.5) edge (-1, -1.5)
	
	(-2, 2.5) edge (-2, 1.5)
	(-2, 1.5) edge (-2, .5)
	(-2, .5) edge (-2, -.5)
	(-2, -.5) edge (-2, -1.5)
	
	(0, 2.5) edge (0, 1.5)
	(0, 1.5) edge (0, .5)
	(0, .5) edge (0, -.5)
	(0, -.5) edge (0, -1.5)

	;

\end{tikzpicture}
\,\,\,
\begin{tikzpicture}[scale = .5]

\vertex (-3, .5) at (-3, .5) {};
\vertex (-3, -.5) at (-3, -.5) {};
\vertex (-3, -1.5) at (-3, -1.5){};
\vertex (-3, 1.5) at (-3, 1.5){};
\vertex (-3, 2.5) at (-3, 2.5){};
\vertex(-3, -2.5) at (-3, -2.5){};
\vertex(-3,- 3.5) at (-3, -3.5) {};

\vertex (-2, .5) at (-2, .5){};
\vertex (-2, -.5) at (-2, -.5){};
\vertex (-2, -1.5) at (-2, -1.5){};
\vertex (-2, 1.5) at (-2, 1.5){};
\vertex (-2, 2.5) at (-2, 2.5){};
\vertex(-2, -2.5) at (-2, -2.5){};
\vertex(-2,- 3.5) at (-2, -3.5) {};

\vertex(-1, .5) at (-1, .5){};
\vertex (-1, -.5) at (-1, -.5){};
\vertex (-1, 1.5) at (-1, 1.5){};
\vertex(-1, -1.5) at (-1, -1.5){};
\vertex (-1, 2.5) at (-1, 2.5){};
\vertex(-1, -2.5) at (-1, -2.5){};
\vertex(-1,- 3.5) at (-1, -3.5) {};

\vertex(0,- 3.5) at (0, -3.5) {};
\vertex(0, .5) at (0, .5){};
\vertex(0, -.5) at (0, -.5){};
\vertex (0, 1.5) at (0, 1.5){};
\vertex(0, -1.5) at (0, -1.5){};
\vertex (0, 2.5) at (0, 2.5){};
\vertex(0, -2.5) at (0, -2.5){};

\vertex(1, .5) at (1, .5){};
\vertex (1, -.5) at (1, -.5){};
\vertex (1, 1.5) at (1, 1.5){};
\vertex(1, -1.5) at (1, -1.5){};
\vertex (1, 2.5) at (1, 2.5){};
\vertex(1, -2.5) at (1, -2.5){};
\vertex(1,- 3.5) at (1, -3.5) {};

\vertex (2, .5) at (2, .5){};
\vertex (2, -.5) at (2, -.5){};
\vertex (2, -1.5) at (2, -1.5){};
\vertex (2, 1.5) at (2, 1.5){};
\vertex (2, 2.5) at (2, 2.5){};
\vertex(2, -2.5) at (2, -2.5){};
\vertex(2,- 3.5) at (2, -3.5) {};

\vertex (3, .5) at (3, .5){};
\vertex (3, -.5) at (3, -.5){};
\vertex (3, -1.5) at (3, -1.5){};
\vertex (3, 1.5) at (3, 1.5){};
\vertex (3, 2.5) at (3, 2.5){};
\vertex(3, -2.5) at (3, -2.5){};
\vertex(3,- 3.5) at (3, -3.5) {};

\path

	(-3, -3.5) edge (-2, -3.5)
	(-2, -2.5) edge (-1, -2.5)
	(-1, -3.5) edge (0, -3.5)
	(0, -2.5) edge (1, -2.5)
	(1, -3.5) edge (2, -3.5)
	(2, -2.5) edge (3, -2.5)
	
	(1, 2.5) edge (1, 1.5)
	(1, 1.5) edge (1, .5)
	(1, .5) edge (1, -.5)
	(1, -.5) edge (1, -1.5)
	(1, -1.5) edge (1, -2.5)
	(1, -2.5) edge (1, -3.5)
	
	(2, 2.5) edge (2, 1.5)
	(2, 1.5) edge (2, .5)
	(2, .5) edge (2, -.5)
	(2, -.5) edge (2, -1.5)
	(2, -1.5) edge (2, -2.5)
	(2, -2.5) edge (2, -3.5)
	
	(3, 2.5) edge (3, 1.5)
	(3, 1.5) edge (3, .5)
	(3, .5) edge (3, -.5)
	(3, -.5) edge (3, -1.5)
	(3, -1.5) edge (3, -2.5)
	(3, -2.5) edge (3, -3.5)
	
	(-1, 2.5) edge (-1, 1.5)
	(-1, 1.5) edge (-1, .5)
	(-1, .5) edge (-1, -.5)
	(-1, -.5) edge (-1, -1.5)
	(-1, -1.5) edge (-1, -2.5)
	(-1, -2.5) edge (-1, -3.5)
	
	(-2, 2.5) edge (-2, 1.5)
	(-2, 1.5) edge (-2, .5)
	(-2, .5) edge (-2, -.5)
	(-2, -.5) edge (-2, -1.5)
	(-2, -1.5) edge (-2, -2.5)
	(-2, -2.5) edge (-2, -3.5)
	
	(-3, 2.5) edge (-3, 1.5)
	(-3, 1.5) edge (-3, .5)
	(-3, .5) edge (-3, -.5)
	(-3, -.5) edge (-3, -1.5)
	(-3, -1.5) edge (-3, -2.5)
	(-3, -2.5) edge (-3, -3.5)
	
	(0, 2.5) edge (0, 1.5)
	(0, 1.5) edge (0, .5)
	(0, .5) edge (0, -.5)
	(0, -.5) edge (0, -1.5)
	(0, -1.5) edge (0, -2.5)
	(0, -2.5) edge (0, -3.5)
	
	;

\end{tikzpicture}

\caption{One construction of $G_2(p\p)$ with $\D = 3$, $p= 16, 24$}
\end{centering}
\end{figure}
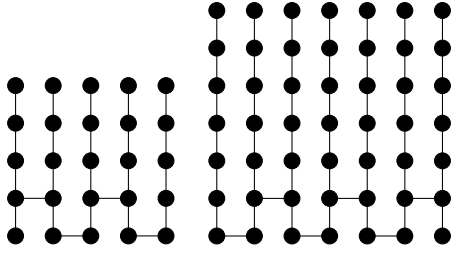

Finally, we consider the case of $\D = 3$, and show that $N^o_k(\D, p)$ and $N^e_k(\D, p)$ are $\Theta(p^k)$.  Since we are ignoring the coefficient of $p^k$, we may combine even and odd cases.  For ease of notation, we will work in $G^e$.

\begin{prop}  For $p$ large enough, there exists $G_k (p)$ in $\Z^k$ with the following properties:
\begin{enumerate}

\item The diameter of $G_k (p)$ is less than or equal to $2p$;

\item The maximum degree of vertices in $G_k (p)$ is 3;

\item $G_k (p)$ has two adjacent vertices of degree less than 3;

\item $|G_k (p)| = \lfloor p/2 \rfloor ^{k-1}(1/4)^{k-1} 2p$.

\end{enumerate}
\label{Gkp}
\end{prop}
\begin{proof}
We start with the base case.  Let $G_1 (p)$ be the induced subgraph on the vertices in the interval $[-p, p]$.  Now assume that there exists $G_{k-1}(p)$ for large enough $p$ satisfying the conditions in Proposition \ref{Gkp}.  We construct $G_k (p)$ as follows.

At $x_{k} = \pm i$ put a copy of $G_{k-1}(p/4)$ for $0 \leq i < p/4$.  By condition 3 of Proposition \ref{Gkp}, at each $x_{k} = \pm i$, there are two vertices, $v^i_1$ and $v^i_2$, that are adjacent and have degree less than $3$.  We choose the same pair on each copy of $G_{k-1}(ap/4)$.  Note that because we chose corresponding pairs $v^i_1$ has the same first $k-1$ coordinates for $- p/4  < i <   p/4$ and similarly for the $v^i_2$, and thus we can connect them with edges in the $k$-th dimension.  If $i$ is even, and less than $p/4-1$ then connect $v^i_1$ to $v^{i+1}_1$.  If $i$ is even and greater than $0$ connect $v^i_2$ to $v^{i-1}_2$.   Figure 7 shows a construction of $G_2(p\p)$.

To check that the degree of vertices $G_k(p)$ is bounded by 3, we note that the only vertices whose degree changed are the vertices $v^i_1$ or $v^i_2$ for $0 \leq i <1/4 p$.  By condition 3 of Proposition \ref{Gkp}, these vertices originally had degree 2, and since we added at most one edge to each, the maximum degree of vertices in $G_k (p)$ is bounded by 3.

Now we compute the diameter of $G_k (p)$.  Let $v,w$ be two vertices in $G_k (p)$ with $v$ located in the plane $x_k = i$ and $w$ located in the plane $x_k = i + j$ for $-p/4 < i < p/4$, $j < p/2$.  First assume $i$ is even.  Since $v$ is located in a copy of $G_{k-1}(p/4)$, by condition 1 of Proposition \ref{Gkp}, $v$ is at most distance $p/2$ from $v^i_1$.  By construction, $v^i_1$ is adjacent to $v^{i+1}_1$ when $i$ is even.  Then, since $i+1$ is odd, $v^{i+1}_1$ is adjacent to $v^{i+1}_2$, which is adjacent to $v^{i+2}_2$.  Finally, $v^{i+2}_2$ is adjacent to $v^{i+2}_1$, and $i+2$ has the same parity as $i$.  Since adjacencies of the $v_1^k$ and $v_2^k$ only matter up to parity, we can continue using this sequence of adjacencies,  we see that the distance from $v^i _1$ to $v^{i+j}_1$ or $v^{i+j}_2$ is at most $2j$, which is at most $p$.  Finally, from $v^{i+j}_1$ to $w$ is at most distance $p/2$.  Therefore the distance from $v$ to $w$ is at most $2p$.  If instead $i$ were odd, the proof follows the same way starting instead with $v^i_2$:  $v$ is at most distance $p/2$ from $v^{i}_2$, which is at most $2j$ from $v^{i+j}_1$ or $v^{i+j}_2$, which is at most $p/2$ from $w$.  Thus the distance in either case from $v$ to $w$ is at most $2p$.

Next we verify $G_k (p)$ has two adjacent vertices of degree less than $3$.  Notice that when $k = 1$, if $p$ is large enough then $G_1(p)$ has at least $2p$ pairs of vertices that are adjacent and have degree less than 2.  Since $G_{k}(p)$ contains around $ p^k/2^k$ copies of $G_1(p)$, and each iteration of this construction uses at most one pair of adjacent free vertices of this copy of $G_1(p)$, for large enough $p$ there will be two adjacent vertices with degree less than 2.

Finally we count the number of vertices in $G_{k} (p)$.  By condition 4 of Proposition \ref{Gkp}, $G_{k-1}(p)$ has  $\lfloor p/2 \rfloor ^{k-2}(1/4)^{k-2} 2p$ vertices.  Since $G_k(p)$ is constructed of $\lfloor p/2 \rfloor$ copies of $G_{k-1}(p/4)$, we have $|G_k (p)| = \lfloor p/2\rfloor ^{k-1}(1/4)^{k-1} 2p$.

This finishes the final case in the proof of Theorem 1 and Theorem 2.
\end{proof}
\section{Conclusions and Open Problems}

Theorem 1 and Theorem 2 narrow the bounds on the size of the largest subgraph in $k$-dimensional mesh with bounded degree and diameter.  Still, several questions remain about MaxDDBS in the mesh.  First, we would like to see a proof or counterexample of the upper bounds in Proposition 1.  While we believe these upper bounds are correct, it is not straightforward to show this.  Furthermore, the bounds in Theorems 1 and 2 could be improved.  We showed that when $\D = 3$, $N^e_k(p, \D)$ and $N^o_k(p, \D)$ are $\Theta (p^k)$.  It seems likely that this could be improved, and with better constructions it could be shown that $N^e_k(p, \D)$ and $N^o_k(p,\D)$ are $\frac{2^k p^k}{k!} + O (p^{k-1})$.  Similarly, when $\D = 4$, we get a lower bound that matches the upper bound the first two terms.  If we let $\D$ be linear in $k$, it seems likely that the lower bounds can be shown to match the upper bounds in even more terms.

Other than the mesh, we can look at MaxDDBS in other host graphs.  In \cite{Dekker} there is some discussion for hypercubes and random networks, but other than that and the case of the mesh, there has been no other work done on MaxDDBS.

\section{Acknowledgements}

This research was conducted at the Duluth REU at the University of Minnesota Duluth, funded by NSF/DMS grant 1062709 and NSA grant H98230-11-1-0224.  I would like to thank to Joe Gallian, and the advisors, Adam Hesterberg, Eric Riedl, and Davie Rolnick for all of their help.  I would also like to thank the visitors and other participants, especially Sam Elder and Ben Kraft.

\end{document}